\documentclass[10pt]{amsart}
\usepackage{amssymb,amsmath,amsrefs}
\usepackage{physics}
\usepackage[hidelinks]{hyperref}
\newtheorem{thm}{Theorem}[section]
\newtheorem{lem}{Lemma}[section]

\newtheorem{cor}{Corollary}[section]

\theoremstyle{definition}

\newtheorem*{questionA}{\textbf{Problem {\sc A}}}
\newtheorem*{questionC}{\textbf{Problem {\sc C}}}
\newtheorem*{questionAA}{\textbf{Problem {\sc A}*}}
\newtheorem*{questionAs}{\textbf{Asymptotic Problem {\sc A}}}
\newtheorem*{questionCs}{\textbf{Asymptotic Problem {\sc C}}}

\def\QQQ {\QQ\cap B(0,1)}
\def\ZZ {\Z\{z\}}

\newcommand{\Z}{\mathbb{Z}}
\newcommand{\Q}{\mathbb{Q}}
\newcommand{\QQ}{\overline{\Q}}

\newcommand{\A}{\mathcal{A}}
\newcommand{\cR}{\mathcal{R}}
\newcommand{\B}{\mathcal{B}}
\newcommand{\s}{\mathcal{S}}

\newcommand{\OO}{\Omega}

\begin{document}
	
	\title{Some results on asymptotic versions of Mahler's problems}
	
	\author[R. Francisco]{Ricardo Francisco}
	\address{Universidade Federal do Agreste de Pernambuco\\
		Garanhuns, PE\\
		Brazil}
	\email{ricardofrancisco628@gmail.com}
	
	\author[D. Marques]{Diego Marques}
	\address{Departamento de Matem\'atica\\
		Universidade de Bras\'ilia\\
		Bras\'ilia, DF\\
		Brazil}
	\email{diego@mat.unb.br}

	
	\begin{abstract}    
		In this paper, we show the existence of a transcendental function $f\in\mathbb{Z}\{z\}$ with coefficients that are almost all bounded such that $f$ and all its derivatives assume algebraic values at algebraic points. Furthermore, we demonstrate that certain subsets of algebraic numbers are exceptional sets of some transcendental function $f\in\mathbb{Z}\{z\}$ with almost all bounded coefficients.
	\end{abstract}
	
	\maketitle

\section{Introduction and main results}\label{intro_section}

The study of the arithmetic properties of numbers of the form \( f(\alpha) \), where \(\alpha\) is an algebraic number and \(f\) is a transcendental analytic function, rose to prominence in the 19th century, notably after Lindemann's proof that \( e^\alpha \) is transcendental for any nonzero algebraic \(\alpha\). In 1886, Strauss attempted to establish that no transcendental analytic function could map \(\mathbb{Q}\) into \(\mathbb{Q}\). However, this conjecture was disproved by Weierstrass, who presented a counterexample.

Weierstrass not only refuted Strauss's conjecture but also went further by conjecturing the existence of a transcendental entire function that maps \(\QQ\) into itself. This conjecture was confirmed in 1895 by Stäckel \cite{stackel1895ueber}, who demonstrated a significantly more general result: for any countable subset \(\Sigma \subseteq \mathbb{C}\) and any dense subset \(T \subseteq \mathbb{C}\), there exists a transcendental entire function \(f\) such that \(f(\Sigma) \subseteq T\). Weierstrass's conjecture is a special case of this result, obtained by setting \(\Sigma = T = \QQ\).

In one of his books, Mahler~\cite[Chap.~3]{bookmahler} posed three foundational problems concerning the arithmetic behavior of transcendental functions, which he labeled Problems A, B, and C. Problems B and C were later completely resolved by Marques and Moreira (see \cite{guguB} and \cite{guguC}), while Problem A remains open. Below, we state two of these problems.

In what follows, \(\mathbb{Z}\{z\}\) denotes the set of power series analytic in \(B(0, 1)\) with integer coefficients, \(\rho_f\) represents the radius of convergence of an analytic function \(f\), and \(S_f\) refers to the \emph{exceptional set} of \(f\), defined as \(\{\alpha \in \QQ \cap \mathrm{dom}(f) : f(\alpha) \in \QQ\}\).

	\begin{questionA}
	\textit{Does there exist a transcendental function $f\in\Z\{z\}$ with bounded coefficients and such that $f(\QQQ)\subseteq\QQ$?}
	\end{questionA}

 \begin{questionC}
\textit{Does there exist for every choice of $S$ (closed under complex conjugation and such that $0\in S$) a transcendental entire function with rational coefficients for which $S_f=S$? }
\end{questionC}

Let $f(z)=\sum_{n\geq 0}a_nz^n$ be a power series and $M\geq 0$, we denote by $L(f,M)$ the set of indexes $n\geq 0$ for which $|a_n|\leq M$, i.e., 
\[
L(f,M)=\{n\in \Z_{\geq 0}: |a_n|\leq M\}.
\]
For ease of notation, we will abbreviate $L_0(f)$ simply as $L(f,0)$.

Thus, Problem A can be rewritten as 
\begin{questionAA}
\textit{Does there exist a transcendental function $f\in\Z\{z\}$ and an integer $M$ such that $L(f,M)=\Z_{\geq 0}$ and $f(\QQQ)\subseteq \QQ$?}
\end{questionAA}

In 1965, Mahler \cite{mahler1965arithmetic} showed that the answer to Problem A is \textit{No} if \( f(z) = \sum_{n \geq 0} a_n z^n \) is a \textit{strongly lacunary function}, that is, if there exist integer sequences \( (s_n)_n \) and \( (t_n)_n \) such that, for all \( n \geq 1 \), the following conditions hold: 
\[
0 \leq t_{n-1} \leq s_n < t_n, \quad a_{s_n} a_{t_n} \neq 0, \quad a_j = 0 \text{ for } j \in [s_n + 1, t_n - 1],
\]
and \( t_n / s_n \to \infty \) as \( n \to \infty \). For a proof of the transcendence of \( f(z) \), see \cite[p.~40]{bookmahler}.

Recently, the authors \cite{francisco2023note} gave a positive answer to Mahler's Problem A for $f\in\Z\{z\}$ with coefficients that are almost all bounded, in the sense of density. More precisely,

\begin{questionAs}
\textit{Does there exist a transcendental function $f\in\Z\{z\}$ and an integer $M$ such that $\delta(L(f,M))=1$ and $f(\QQQ)\subseteq\QQ$?}
\end{questionAs}

Here, as usual, $\delta(\A)$ denotes the \textit{natural density} of a set $\A\subseteq \Z_{\geq 0}$, i.e., \(\delta(\A):=\lim_{x\to\infty}\#\A(x)/x\) (if the limit exists), where $\A(x):=\A\cap [0,x]$, for $x>0$. More precisely, they proved that for any infinite sets of non-negative integers $\A$ and $\B$, there exist transcendental analytic functions $f\in\Z\{z\}$ whose coefficients vanish for any indexes $n\not\in\A+\B:=\{a+b: (a,b)\in \A\times \B\}$ and such that $f(\QQQ)\subseteq \QQ$.

The first purpose of this work is to establish a stronger result by showing the existence of a transcendental function $f\in\Z\{z\}$ whose coefficients vanish for any indexes $n\not\in\A+\B$ and such that $f$ and all its derivatives assume algebraic values at algebraic points. 

\begin{thm}\label{main1}
Let $\A$ and $\B$ be infinite sets of non-negative integers and set $\mathcal{S}=\A+\B$. Then there exist uncountable many transcendental functions $f(z)=\sum_{n\in \s}a_nz^n\in \Z\{z\}$ such that 
\[
f^{(m)}(\QQQ)\subseteq \QQ,\ \mbox{for all}\ m\geq 0.
\]
\end{thm}

We have the following two consequences:

\begin{cor}\label{cor:derivativesMoreira}
Let $\s$ be a set of positive integers such that $\delta(\s)>0$, then there exist uncountable many transcendental functions $f(z)=\sum_{n\in \s}a_nz^n\in\Z\{z\}$ such that 
\[
f^{(m)}(\QQQ)\subseteq\QQ,\ \mbox{for all}\ m\geq 0.
\]
\end{cor}

The previous result follows from Theorem \ref{main1} together with the proof of the Erd\"{o}s' $A+B$ conjecture: in fact, Moreira {\it et al.} \cite{moreira} showed that any $\s\subseteq \Z_{\geq 0}$, for which $\delta(\s)>0$, contains a sumset of two infinite sets.

\begin{cor}\label{cor:derivatives}
There exist uncountable many transcendental functions $f\in\Z\{z\}$ such that 
\[
f^{(m)}(\QQQ)\subseteq\QQ\quad\mbox{and}\quad \delta(L_0(f^{(m)}))=1,\ \mbox{for all}\ m\geq 0.
\]
\end{cor}

To prove this, we use Theorem \ref{main1} with the choice of $\A=\B=\Z_{\geq 0}^2$ and by using the fact that $L_0(f^{(m)})=(L_0(f)\cap [m,+\infty))$.

Since the set of prime numbers $\mathbb{P}$ has asymptotic density equal to 0, Corollary \ref{cor:derivativesMoreira} cannot be directly applied. However, we can overcome this difficulty by using a recent result by Tao and Ziegler (see \cite{tao2023infinite}). We present the theorem in a form free of derivatives, though a derivative-based version analogous to Theorem \ref{main1} is also possible.

\begin{thm}\label{main2}
Let $\A=\{a_1,a_2,\ldots\}$ and $\B=\{b_1, b_2,\ldots\}$ be infinite sets of nonnegative integers, listed in increasing order. Then there exist uncountable many transcendental functions 
\[
f(z)=\sum_{n\in\s}c_nz^n\in\Z\{z\}
\]
such that $f(\QQQ)\subseteq\QQ$. Where $\s:=\{a_i+b_j\in\A+\B: 1\leq i<j\}$ is referred to as the \textit{partial sum} of $\A$ and $\B$.
\end{thm}

\begin{cor}\label{cor:primes}
There exist uncountably many transcendental functions 
\[
f(z)=\sum_{n\in\mathbb{P}}a_nz^n\in\Z\{z\},
\]
such that $f(\QQQ)\subseteq\QQ$. 
\end{cor}

Indeed, the theorem of Tao and Ziegler ensures that $\mathbb{P}$ contains a partial sumset of two infinite sets. Thus, by applying Theorem \ref{main2} to these sets, then yields Corollary \ref{cor:primes}.

Related to the Problem {\sc C}, in another work, Marques and Moreira \cite{marques2020exceptional} proved that any subset of $\QQQ$ (closed under complex conjugation and having the element $0$) is the exceptional set of a transcendental function in $\Z\{z\}$. This naturally leads us to pose a question analogous to Asymptotic Problem A: what about Mahler's Problem C over $\mathbb{Z}$ when considering coefficients that are almost all bounded? In other words,

\begin{questionCs}
\textit{Does there exist for any choice of $S\subseteq \QQQ$ (closed under complex conjugation and such that $0\in S$) a transcendental function $f\in\Z\{z\}$ and an integer $M\geq 0$ such that $\delta(L(f,M))=1$ and $S_f=S$?}
\end{questionCs}

In this work, we present several partial results addressing this question.

To this end, let $\rho \in (0, \infty]$ and $S \subseteq \QQ \cap B(0, \rho)$. Denote by $\overline{S}^{\mathrm{alg.}}$ the set of all algebraic conjugates of the elements in $S$. We say that \textit{the set $S$ is closed relative to $\QQ$} if 
\[
\overline{S}^{\rm alg.}\cap B(0,\rho)=S.
\]
	
Mahler \cite{mahler1965arithmetic} demonstrated that certain classes of transcendental analytic functions, specifically strongly lacunary functions, possess exceptional sets that are closed with respect to $\QQ$.

Within this context, we prove the following results.

	\begin{thm}\label{main3}
		Let $\rho\in(0,1]$ be a real number. If $S$ is a subset of $\QQ\cap B(0,\rho)$ which is closed relative to $\QQ$, with $0\in S$, then there exist uncountably many transcendental functions $f(z)\in\ZZ$, such that 
        \begin{center}
        $S_f=S$,\ \  $\rho_f=\rho$\ \  and\ \  $\delta(L_0(f))=1$.
        \end{center}
	\end{thm}

    In the next results, we denote $\cR_{P}$ as the set of complex roots of a polynomial $P(x)$, $\mu_k$ as the set of $k$-th roots of unity and for $\A \subseteq \mathbb{Z}_{\geq 0}$, $\underline{\delta}(\A)$ as the \textit{lower asymptotic density} of $\A$, defined by $\liminf_{n \to \infty} \#\A(n)/n$.

    Also, to avoid overloading the notation, for a nonconstant polynomial $P(z)$, with $P(0)=0$, we set $\OO_P:=P^{-1}(B(0,1))\cap B(0,1)$. Observe that $\OO_P$ is a non-empty open set. In particular, $\OO_p\cap \QQ$ is a dense subset of $\OO_p$. 

\begin{thm}\label{main4}
Let $P(z) = \sum_{i=1}^k a_{m_i} z^{m_i} \in \mathbb{Z}[z]$ be a non-constant polynomial, where $a_{m_i} \neq 0$ for all $i \in [1, k]$, and $P(0) = 0$. Let $S \subseteq \OO_P\cap \QQ$ be a set that is closed under complex conjugation, with $0 \in S$ and the property that 
\begin{center}
   $\mathcal{R}_{P(z) - P(\alpha)} \subseteq S$, for all $\alpha \in S$. 
\end{center}
Then, there are uncountably many transcendental functions $\psi(z)$, with integer coefficients and analytic in $\OO_P$, such that
\[
S_{\psi}=S
\]
and 
\[
\underline{\delta}(L_0(\psi))\geq 1-\dfrac{1}{\gcd(m_1,\ldots,m_k)},
\]
where we adopt the convention that $\gcd(m_1, \ldots, m_k) = m_1$ when $k = 1$.
\end{thm}

As an immediate consequence, we have

\begin{cor}\label{cor:unity}
Let $\ell$ be a positive integer, and let $S \subseteq \QQQ$ be a set that is closed under complex conjugation, contains $0$, and such that $\mu_{\ell} \cdot S \subseteq S$. Then, there exist uncountably many transcendental functions $\psi(z) \in \Z\{z\}$ such that
\begin{center}
$S_{\psi}=S$\ \ \  and\ \ \  $\underline{\delta}(L_0(\psi))\geq 1-1/\ell$.
\end{center}
Furthermore, if $\mu_{k} \cdot S \subseteq S$ for infinitely many positive integers $k$, then these functions $\psi$ can be chosen such that $\delta(L_0(\psi)) = 1$.
\end{cor}

	\section{The proofs of theorems \ref{main1} and \ref{main2}}
	\subsection{A key lemma} 
    Before the proofs, we present a key result that will serve as an essential component for what follows. This result was previously established in \cite{francisco2023note}, but for completeness, we include its proof here.
	\begin{lem}\label{lem:1}
		Let $P(z)\in \Z[z]$ be a polynomial of degree $d\geq 1$ and $\mathcal{S}$ an infinite set of positive integers. Then there exists a nonzero $m$-degree polynomial $Q(z)\in\mathbb{Z}[z]$ such that the polynomial $PQ\in\Z[z]$ has the form 
		\[
		\sum_{\substack{n\in\mathcal{S}(d+m)}}a_nz^n.
		\] 
	\end{lem}
	\begin{proof}
		Set $Q(z)=\sum_{i=0}^Lq_iz^i$, where $L$ and the coefficients $q_i$ (for $i\in [0,L]$) will be appropriately chosen. Note that the polynomial $PQ$ has degree at most $L+d$ and its coefficients are linear forms in $q_0,\ldots, q_L$. So, it suffices to prove that it is possible to choose the coefficients of $Q(z)$ in order to eliminate the terms $z^n$ in $(PQ)(z)$, for which $n\notin\mathcal{S}(L+d)$ (observe that there are $L+d+1-\#\mathcal{S}(L+d)$ such terms). By setting the coefficients of \(PQ\) to zero, we obtain a homogeneous linear system with \(L+d+1-\#\mathcal{S}(L+d)\) equations and \(L+1\) variables \(q_i\) (\(i \in [0, L]\)). Thus, this system admits a non-trivial integer solution $(q_0,\ldots,q_L)$ provided that
		\[
		L+1>L+d+1-\#\mathcal{S}(L+d),
		\]
		that is, if $\#\mathcal{S}(L+d)>d.$ Since $\mathcal{S}$ is an unbounded set, the previous inequality holds for all sufficiently large integers $L$. Thus, if $m:=\max\{i\in [0,L]: q_i\neq 0\}$, then  
		\[
		Q(z)=\sum_{i=0}^mq_iz^i\in\mathbb{Z}[z]
		\]
		is the desired polynomial. This completes the proof.
	\end{proof}

\subsection{Proof of Theorem \ref{main1}}
Let $\{\alpha_1, \alpha_2,\ldots\}$ be an enumeration of $\QQQ$ and let $P_i(z)$ be the minimal polynomial (over $\Z$) of the algebraic number $\alpha_i$ of degree $d_i$. By setting $B_k(z):=(P_1(z)\cdots P_k(z))^k$ and since $\B$ is an infinite set of positive integers, Lemma \ref{lem:1} ensures the existence (for each $k\geq 1$) of an $m_k$-degree polynomial $Q_k(z)\in\Z[z]$ for which
\[
Q_k(z)B_k(z)=\sum_{n\in\B(m_k+D_k)}a_{k,n}z^n,
\] 
where $D_k=k\cdot \sum_{i=1}^kd_i$ is the degree of $B_k$. Next, we define the recurrence sequence $(t_k)_{k}$ by $t_1=\min (\A)$ and with two possible choices for $t_{k+1},$ namely, $t_{k+1}\in \{v_k, w_k\}$, where $v_k=\min E_k, w_k=\min(E_k\backslash\{v_k\})$ and 
\[
E_k:=\{t\in \A : t\geq \max\{k(t_k+D_k+m_k)+1, L(Q_{k+1}B_{k+1})+(k+1)\}.
\]
As before, this choice is possible because $\A$ is an infinite set of positive integers. Here, as usual, $L(P)$ denotes the \textit{length} of the polynomial $P$ (i.e., the sum of the absolute values of its coefficients). 

We claim that the function
\[
f(z):=\sum_{k\geq 1}z^{t_k}Q_k(z)B_k(z)
\]
satisfies the condition of the statement. Indeed, first we note that by construction, the function $f(z)$ can be written as $\sum_{n\in\s}a_nz^n$ (as $t_{k+1}>t_k+D_k+m_k$). Moreover, since $t_{k+1}/(t_k+D_k+m_k)$ tends to infinity as $k\to\infty$, then $f(z)$ is a strongly lacunary series and thus a transcendental function. 

Now, we shall prove that $f$ is an analytic function in the unit ball. For that, take $R\in(0,1)$, $z\in\bar{B}(0,R)$ and using that $\abs{P(z)}\leq L(P)$ when $\abs{z}\leq 1$, one infers that
\[
\abs{z^{t_k}Q_k(z)B_k(z)}\leq R^{t_k}L(Q_kB_k)<R^{L(Q_kB_k)+k}L(Q_kB_k),
\]
where we used that $t_k>L(Q_kB_k)+k$ and $R<1$. 

Note now that the maximum value of the function $x\mapsto xR^x,$ for real positive values of $x$, is attained at $x=1/\abs{\log R}$, and is equal to $e^{-1}/\abs{\log R}$. This implies that 
\[
R^{L(Q_kB_k)+k}L(Q_kB_k)<\frac{e^{-1}}{\abs{\log R}}R^k.
\] 
Summarizing, we get \[\abs{z^{t_k}Q_k(z)B_k(z)}\leq \frac{e^{-1}}{\abs{\log R}}R^k=: M_k,\] for all $z\in\bar{B}(0,R)$. Since \(\sum_{k \geq 1} M_k\) converges, the Weierstrass \(M\)-test ensures that the series 
\[
\sum_{k \geq 1} z^{t_k} Q_k(z) B_k(z)
\] 
converges absolutely and uniformly on \(\overline{B}(0, R)\) for any \(R \in (0, 1)\). Consequently, this series defines an analytic function, \(f(z)\), within the unit ball \(B(0, 1)\).

This series, in particular, can be differentiated term by term an arbitrary number of times. Thus, for each $m\geq 0$, we have 
\begin{eqnarray*}
	f^{(m)}(z)&=&\sum_{k=1}^m\frac{d^m}{dz^m}[z^{t_k}Q_k(z)B_k(z)]+\sum_{k\geq m+1} \frac{d^m}{dz^m}[z^{t_k}Q_k(z)B_k(z)]\\&=&\sum_{k=1}^m\frac{d^m}{dz^m}[z^{t_k}Q_k(z)B_k(z)]+\sum_{k\geq m+1}\sum_{j=0}^m\binom{m}{j}\frac{d^{m-j}}{dz^{m-j}}[z^{t_k}Q_k(z)]B_k^{(j)}(z).
\end{eqnarray*}

By definition, for integers $i\geq 1$ and $m\geq 0$, we have
\[
B_k^{(j)}(\alpha_i)=0,
\]
for all $j\in [0, m]$, $k\geq m+1$ and $i\in [1,k]$. Thus, $f^{(m)}(\alpha_i)$ is expressed as a finite sum of algebraic numbers and is therefore algebraic. Therefore, 
\[
f^{(m)}(\QQQ)\subseteq \overline{\mathbb{Q}},\;\mbox{for all}\;m\geq 0.
\]

Finally, for any $k>1, t_k$ can be chosen in two different ways and any of these choices provide different functions $f$. Thus, we have constructed uncountably many of these of these functions, completing the proof.
\qed


\subsection{Proof of Theorem \ref{main2}} As before, let $\{\alpha_1, \alpha_2, \ldots\}$ be an enumeration of $\QQQ$, and let $P_i(z)$ be the minimal polynomial (over $\Z$) of the $d_i$-degree algebraic number $\alpha_i$. Since $\A$ is infinite, by Lemma \ref{lem:1}, for each $k \geq 1$, there exists a polynomial $Q_k \in \Z[z]$ of degree $m_k$ such that
\[
Q_k(z)P_1(z)\cdots P_k(z)=\sum_{n\in\A(m_k+D_k)}c_{k,n}z^n, 
\]
where $D_k:=\sum_{i=1}^kd_i$. Let us set $N_k:=\#\A(m_k+D_k)$. 

We define recursively the sequence $(t_k)_k$ of elements of $\B$ such that $t_1=b_{N_1+1}$ and $t_{k+1}:=b_{\max\{J_k, N_{k+1}+1\}}$, where 
\[
J_k:=\min \{j\ge 1 : b_j\geq \max\{k(t_k+D_k+m_k)+1, L(Q_{k+1}P_1\cdots P_{k+1})+(k+1)\} \}.
\]

We assert that the function
\[
f(z):=\sum_{k\geq 1}z^{t_k}Q_k(z)P_1(z)\cdots P_k(z)=\sum_{n\in\s}c_nz^n
\]
satisfies the conditions of the statement. To confirm that \( f \) is a transcendental analytic function satisfying \( f(\QQQ) \subseteq \QQ \), we follow a similar approach to the proof of Theorem~\ref{main1}.

Furthermore, for each $k\geq 1$, we have at least two possibilities for $t_{k+1}$. In fact, 
\begin{itemize}
	\item if  $\max\{b_{J_k}, b_{N_{k+1}+1}\}=b_{J_k}$, we can take $t_{k+1}\in\{b_{J_k}, b_{J_k+1}\}$ and
	\item  if $\max\{b_{J_k}, b_{N_{k+1}+1}\}=b_{N_{k+1}+1}$, we can take $t_{k+1}\in\{b_{N_{k+1}+1},b_{N_{k+1}+2}\}$.
\end{itemize} 
Thus, we have successfully constructed an uncountable family of such functions. \qed

\section{Proof of the Theorem \ref{main3}}

For this proof, we require the following auxiliary result (proved by Mahler in \cite[Theorem I]{mahler1965arithmetic}).

\begin{lem}\label{Mahler1} Let \( f(z) = \sum_{n \geq 0} a_n z^n \) be a power series with integer coefficients and radius of convergence \( \rho_f \). Assume that \( f \) is a strongly lacunary function and let $\alpha\in\QQ\cap B(0,\rho_f)$. Then $f(\alpha)$ is an algebraic number if and only if there exists a positive integer $N=N(\alpha)$ such that $F_n(\alpha)=0$ for all $n\geq N$, where 
 \[
 F_n(z):=\sum_{k=s_n}^{r_{n+1}}a_kz^k.
 \]
	\end{lem}
 
	\subsection{The proof} We may assume that \( S = \{\alpha_1, \alpha_2, \dots\} \) is an infinite set (the finite cases can be treated similarly). Define a sequence of polynomials \( (P_n(z))_n \), where \( P_k(z) \) is the minimal polynomial (over \( \mathbb{Z} \)) of the algebraic number \( \alpha_k \), which has degree \( d_k \).

Now, let $D_k := \sum_{i=1}^k d_i$ and define the sequence \( (t_k)_k \) recursively as follows: $t_1 = 0$ and with two choices for $t_{k+1},$ specifically, $t_{k+1}\in\{s_k, s_k+1\}$, where

\[
s_k:=\max\{k(t_{k}+D_{k})+1, L(P_1\cdots P_{k+1})+(k+1), 2(k+1)^2D_{k+1}\}.
\]
Since $H(P_1\cdots P_k)^{1/t_k}\leq t_k^{1/t_k}$ (as $H(P_i)\leq L(P_i)$), we have
	\begin{equation}\label{eq 1}
		\lim_{k\to\infty}H(P_1\cdots P_k)^{1/t_k}=1.
	\end{equation}
    Finally, by defining the sequence $(c_k)_k$ by $c_k=\lfloor \rho^{-t_k} \rfloor$. Then, it follows that \begin{equation}\label{eq 2}
		\lim_{k\to\infty}c_k^{1/t_k}=\frac{1}{\rho}. 
	\end{equation}
	We claim that 
    \[
    f(z):=\sum_{k\geq 1}c_kz^{t_k}P_1(z)\cdots P_k(z)=\sum_{k\geq 1}F_k(z)=\sum_{n=0}^{\infty}a_nz^n
    \]
    is the required function. In fact, the radius of convergence $\rho_f$ is positive, because 
    \[
    \frac{1}{\rho_f}=\limsup_{n\to\infty} \abs{a_n}^{1/n}=\limsup_{\substack{t_k\leq n\leq t_k+D_k\\ k\to\infty}}\abs{a_n}^{1/t_k}.
    \]
    In addition, 
    \[
    \abs{a_n}\leq c_nH(P_1\cdots P_n)\quad\mbox{para}\quad t_k\leq n\leq t_k+D_k
    \]
    and the equality holds for some $n$ in this interval. Hence, by \hyperref[eq 1]{(\ref*{eq 1})} and \hyperref[eq 1]{(\ref*{eq 2})} \[
    \frac{1}{\rho_f}=\limsup_{k\to\infty}(c_kH(P_1\cdots P_k))^{1/t_k}=\frac{1}{\rho},
    \]
    and so \(\rho_f=\rho>0\). Furthermore, the power series $f$ is strongly lacunary, because 
    \[
    \lim_{k\to\infty}\frac{t_{k+1}}{t_k+D_k}=\infty.
    \]
    yielding the transcendence of $f$. 
    
  Now, we shall prove that $S_f = S$. For that, if $\alpha_i \in S$, then 
\[
f(\alpha_i) = \sum_{k=1}^{i-1} F_k(\alpha_i) \in \QQ.
\]
Conversely, if $\alpha \not\in S$, it follows that $F_k(\alpha) \neq 0$ for all $k$. Therefore, by Lemma \ref{Mahler1}, we conclude that $f(\alpha) \not\in \QQ$.

	 The next step is to prove that $\delta(L_0(f))=1$. Note that if $a_n\neq 0$ then 
     \[
     n\in\mathcal{A}:=\bigcup_{k\geq 1}\{t_k,t_k+1,\ldots, t_k+D_k\}.
     \]
     By the definition of $t_k$, we have 
     \[
     \frac{\#\mathcal{A}(t_k+D_k)}{t_k+D_k}\leq \frac{2kD_k}{2k^2D_k}=\frac{1}{k},
     \]
     which yields
	  \begin{equation}\label{eq 3}
	 	\lim_{k\to\infty}\frac{\#\mathcal{A}(t_k+D_k)}{t_k+D_k}=0.
	 \end{equation}  
 Moreover, 
 \[
 \frac{\#\mathcal{A}(t_k)}{t_k}=\frac{1}{t_k}\left(1+(k-1)+\sum_{i=1}^{k-1}D_i\right)\leq \frac{2kD_k}{2k^2D_k}=\frac{1}{k}.
 \]
 
 Thus 
 \begin{equation}\label{eq 4}
 	\lim_{k\to\infty}\frac{\#\mathcal{A}(t_k)}{t_k}=0. 
 \end{equation}
By combining (\ref{eq 3}) and (\ref{eq 4}), one infers that $\#\mathcal{A}(n)/n$ tends to $0$ as $n\to \infty$. In fact, the definition of $\A$ implies that for sufficiently large $n$, $n$ will belong to the interval $(t_k+D_k, t_{k+1})$ or $[t_k, t_k+D_k]$. These cases allow us to analyze $\#\A(n)/n$ separately:

\bigskip
\noindent\textbf{Case 1:} $n\notin \mathcal{A}$. In this case, $n\in (t_k+D_k, t_{k+1})$ for some $k$. Then,

\[
\frac{\#\mathcal{A}(n)}{n}=\frac{\#\mathcal{A}(t_k+D_k)}{n}<\frac{\#\mathcal{A}(t_k+D_k)}{t_k+D_k}.
\]

\noindent \textbf{Case 2:} $n\in\mathcal{A}$. Here, we have $n\in [t_k, t_k+D_k]$, for some $k$. Then,
\begin{eqnarray*}
	\frac{\#\mathcal{A}(n)}{n}&\leq& \frac{\#\mathcal{A}(t_k)+D_k}{n}\leq \frac{\#\mathcal{A}(t_k)+D_k}{t_k} <\frac{\#\mathcal{A}(t_k)}{t_k}+\frac{1}{2k^2},
\end{eqnarray*}
where we used that $t_k\geq 2k^2D_k$.

Using (\ref{eq 3}) and (\ref{eq 4}), we conclude that 
\[
\lim_{n\to \infty}\frac{\#\A(n)}{n}=0
\]
and therefore $\delta(\A)=0.$  Thus, since \(\mathbb{Z}_{\geq 0} \setminus \mathcal{A} \subseteq L_0(f)\), it follows that 
\[
\delta(L_0(f)) = 1.
\]

To conclude the proof, we construct an uncountable family of such functions by employing a similar reasoning as presented in the proofs of theorems \ref{main1} and \ref{main2}, thereby establishing the desired result. \qed

\section{The proofs of Theorem \ref{main4} and Corollary \ref{cor:unity}}

\subsection{Proof of the Theorem \ref{main4}} Observe that $P(S) \cap B(0,1)$ is a subset of $\QQQ$, closed under complex conjugation and including the origin, $z=0$. Thus, Theorem 1 of \cite{marques2020exceptional} ensures the existence of an uncountable set, say $\mathcal{E}\subseteq \mathbb{Z}\{z\}$, of transcendental functions such that $S_f=P(S) \cap B(0,1)$, for all $f\in \mathcal{E}$. 

Now, for each $f \in \mathcal{E}$, we define $\psi_f$ as 
\[
\psi_f(z) := f(P(z)),\ \forall z\in \OO_P=\emph{dom}(\psi_f).
\]
It is important to observe that the mapping $f \mapsto \psi_f$ is injective. Indeed, let $f, g \in \mathcal{E}$ such that $\psi_f(z) = \psi_g(z)$ for all $z \in \Omega_P$. Since $\Omega_P$ is a nonempty open set, there exists an infinite compact set $K \subseteq \Omega_P$ for which $f(P(z)) = g(P(z))$ for all $z \in K$. Consequently, $P(K)$ is an infinite compact set with an accumulation point. Thus, $f(w) = g(w)$ for all $w \in P(K)$, and by the identity theorem for analytic functions, we deduce that $f(z) = g(z)$ for all $z \in B(0,1)$. This establishes the injectivity.

Next, we claim that $S$ is the exceptional set of $\psi_f$ for any $f \in \mathcal{E}$. Indeed, if $\alpha \in S$, then $P(\alpha) \in P(S) \cap B(0,1) = S_f$ (since $S \subseteq \Omega_P$), and hence $\psi_f(\alpha) = f(P(\alpha)) \in \QQ$, yielding $S \subseteq S_{\psi_f}$.

Conversely, let $\beta \in (\Omega_P \cap \QQ) \setminus S$. Then $P(\beta) \not\in P(S)$, because if $P(\beta) = P(\gamma)$ for some $\gamma \in S$, it would follow that $\beta \in \mathcal{R}_{P(z) - P(\gamma)} \subseteq S$, a contradiction. Thus, $P(\beta) \not\in P(S) \cap B(0,1) = S_f$, which implies $\psi_f(\beta) = f(P(\beta))$ is a transcendental number. In conclusion, 
\[
S_{\psi_f} = S.
\]

Furthermore, the polynomial $P(z)$ can be expressed as $P(z) = Q(z^d)$, where $d = \gcd(m_1, \ldots, m_k)$. Thus, if $f(z) = \sum_{i \geq 0} c_i z^i \in \mathcal{E}$, we have
\[
\psi_f(z)=\sum_{i\geq 0}c_iQ(z^d)^i=\sum_{j\geq 0}b_{dj}z^{dj}.
\]
Consequently, we get 
\[
\underline{\delta}(L_0(\psi_f)\geq \delta(\Z\backslash d\Z)=1-1/d.
\]

The proof that $\psi_f$ is a transcendental function stems from the fact that the composition of a transcendental function with an algebraic function remains transcendental. Alternatively, we can guarantee the choice of $\psi_f$ as transcendental through the uncountability of the set $\{\psi_f : f \in \mathcal{E}\}$. Since the set of algebraic functions in $\ZZ$ is countable, uncountably many transcendental functions can be chosen for $\psi_f$.
\qed

\subsection{Proof of the Corollary \ref{cor:unity}}

We aim to apply Theorem \ref{main4}. To this end, let us take \( P(z) = z^{\ell} \). In this case, we observe that \(\OO_P = B(0,1)\), which holds if and only if \( P(z) = \pm z^{\ell} \) for some \(\ell \in \mathbb{Z}_{>0}\). Let \( S \) be a set that satisfies the conditions in the corollary statement. Hence, it suffices to show that 

\[
\mathcal{R}_{z^{\ell} - \alpha^{\ell}} \subseteq S, \quad \text{for all } \alpha \in S.
\]
Indeed, if \(\gamma \in \mathcal{R}_{z^{\ell} - \alpha^{\ell}}\), then by definition \(\gamma^{\ell} = \alpha^{\ell}\). Consequently, there exists \(\zeta \in \mu_{\ell}\) such that \(\gamma = \zeta\alpha\). Since \( \mu_{\ell} \cdot S \subseteq S \), it follows that \(\gamma \in S\). This completes the proof.
\qed

\subsection*{Acknowledgements}
This research was conducted during a productive and enjoyable visit by D.M. to IMPA (Rio de Janeiro), where he greatly benefited from its outstanding working conditions. The authors are grateful for the financial support provided by the National Council for Scientific and Technological Development (CNPq).


\begin{thebibliography}{99}
    
		\bibitem{francisco2023note} R. Francisco, D. Marques, A note on an asymptotic version of a problem of Mahler, {\it Bull. Austral. Math. Soc.} {\bf 107(3)} (2023), 398--402. 

		
		\bibitem{bookmahler} K. Mahler, \textit{Lectures on Transcendental Numbers}, Lecture Notes in Math., \textbf{546}, Springer-Verlag, Berlin, 1976.

           \bibitem{mahler1965arithmetic} K. Mahler, Arithmetic properties of lacunary power series with integral coefficients, {\it J. Austral. Math. Soc.} {\bf 5(1)} (1965), 56--64.
		 
		\bibitem{guguB} D. Marques, C. G. Moreira, A positive answer for a question proposed by K. Mahler, {\it Math. Ann.} \textbf{367} (2017), 1059--1062.
		
		\bibitem{guguC} D. Marques, C. G. Moreira, A note on a complete solution of a problem posed by Mahler, {\it Bull. Austral. Math. Soc.} \textbf{98} (2018), 60--63.
		
		\bibitem{marques2020exceptional} D. Marques, C. G. Moreira, On exceptional sets of transcendental functions with integer coefficients: solution of a problem of Mahler, {\it Acta Arith.} \textbf{192} (2020), 313--327.
		
		
		\bibitem{moreira} J. Moreira, F. Richter, D. Robertson, A proof of a sumset conjecture of Erd\"{o}s. {\it Ann. of Math.} {\bf 189} (2019), 605--652.

        \bibitem{stackel1895ueber} P. St{\"a}ckel, Ueber arithmetische Eigenschaften analytischer Functionen,. {\it Math. Ann.} {\bf 46(4)} (1895), 513--520.
		
		\bibitem{tao2023infinite} T. Tao, T. Ziegler, Infinite partial sumsets in the primes. {\it J. Analyse Math.} {\bf 151} (2023), no. 1, 375--389.	
		
		
		

       
		
		
	\end{thebibliography}
\end{document}